\documentclass{amsart}
\usepackage{amsmath, amssymb}
\usepackage{color, mathdots}
\usepackage{url}

\newcommand*{\claimproofname}{Proof of Claim}
\newenvironment{claimproof}[1][\claimproofname]{\begin{proof}[#1]}{\end{proof}}

\input xy
\xyoption{all}

\newtheorem{theorem}{Theorem}[section]
\newtheorem{lemma}[theorem]{Lemma}
\newtheorem{corollary}[theorem]{Corollary}
\newtheorem{fact}[theorem]{Fact}
\newtheorem{proposition}[theorem]{Proposition}
\newtheorem{claim}[theorem]{Claim}

\theoremstyle{definition}
\newtheorem{example}[theorem]{Example}
\newtheorem{question}[theorem]{Question}
\newtheorem{remark}[theorem]{Remark}
\newtheorem{definition}[theorem]{Definition}

\def\K{\mathcal K}

\def\bider{\{\cdot,\cdot\}}
\def\id{\operatorname{id}}
\def\Frac{\operatorname{Frac}}
\def\Spec{\operatorname{Spec}}
\def\alg{\operatorname{alg}}

\def\sep{\operatorname{sep}}

\def\Ind#1#2{#1\setbox0=\hbox{$#1x$}\kern\wd0\hbox to 0pt{\hss$#1\mid$\hss}
\lower.9\ht0\hbox to 0pt{\hss$#1\smile$\hss}\kern\wd0}

\def\Notind#1#2{#1\setbox0=\hbox{$#1x$}\kern\wd0\hbox to 0pt{\mathchardef
\nn=12854\hss$#1\nn$\kern1.4\wd0\hss}\hbox to
0pt{\hss$#1\mid$\hss}\lower.9\ht0 \hbox to
0pt{\hss$#1\smile$\hss}\kern\wd0}

\begin{document}

\title[Commutative bidifferential algebra]{Commutative
bidifferential algebra}

\author{Omar Le\'on S\'anchez}
\address{Department of Mathematics, University of Manchester, Oxford Road, Manchester, United Kingdom M13 9PL}
\email{omar.sanchez@manchester.ac.uk}

\author{Rahim Moosa}
\address{Department of Pure Mathematics, University of Waterloo, 200 University Avenue West, Waterloo, Ontario, Canada N2L 3G1}
\email{rmoosa@uwaterloo.ca}

\date{\today}
\thanks{{\em Acknowledgements}: R. Moosa was partially supported by an NSERC Discovery Grant.}

\begin{abstract}
Motivated by the Poisson Dixmier-Moeglin equivalence problem, a systematic study of commutative unitary rings equipped with a {\em biderivation}, namely a binary operation that is a derivation in each argument, is here begun, with an eye toward the geometry of the corresponding {\em $B$-varieties}.
Foundational results about extending biderivations to localisations, algebraic extensions and transcendental extensions are established.
Resolving a deficiency in Poisson algebraic geometry, a theory of base extension is achieved, and it is shown that dominant $B$-morphisms admit generic $B$-fibres.
A bidifferential version of the Dixmier-Moeglin equivalence problem is articulated.
\end{abstract}

\maketitle

\setcounter{tocdepth}{1}
\tableofcontents

\section{Introduction}

\noindent
In this paper we begin a systematic study of commutative {\em bidifferential} rings; that is, commutative rings equipped with a binary function $\bider:R\times R\to R$ which is a derivation in each argument.
We are primarily interested in the case when $R$ is a finitely generated integral algebra over a field $k$ -- so that we are really talking about affine algebraic varieties equipped with some additional structure, what we call here {\em $B$-varieties}\footnote{In this terminology, and indeed in our general approach, we are informed by the theory of $D$-varieties as introduced by Buium~\cite{Buium}.}.
In fact, the motivating examples are {\em Poisson varieties}, namely when $\bider$ is in addition $k$-bilinear, skew symmetric, and satisfies the Jacobi identity.
But from our point of view, it is more natural to drop these additional assumptions on $\bider$, at least for now.

In order to illustrate our approach, and the benefit of working beyond the Poisson setting, consider the following question: Given a dominant morphism of affine Poisson varieties $\phi:X\to Y$ over a field $k$ (of characteristic zero) can we make sense of the generic fibre of $\phi$ as a Poisson variety?
The generic fibre is obtained by taking a base extension to $L:=k(Y)$, namely $\phi_L:X_L\to Y_L$, and then considering the fibre $\phi_L^{-1}(\alpha)$ where $\alpha\in Y_L(L)$ is a generic point over $k$.
Naively, one might expect that there be canonical Poisson structures on $X_L$ and $Y_L$, that $\phi_L$ will be a Poisson morphism, and that $\phi^{-1}(\alpha)$ will be a Poisson subvariety of $X_L$.
In fact, all three of these expectations fail separately.
One of the main technical accomplishments of this paper (Theorem~\ref{genPfibre} below) is to find a particular bidifferential structure on $X_L$ and $Y_L$, compatible with the original Poisson structures, such that $\phi_L$ is a morphism of $B$-varieties and $\phi_L^{-1}(\alpha)$ is a $B$-subvariety of $X_L$.
The bidifferential structure on $X_L$ is not canonical in the sense that it depends on $\phi$, and not just on the Poisson structure on $X$ and $L$.
Nor is the biderivation we construct necessarily a Poisson bracket.
Nevertheless, by working in the more flexible category of $B$-algebraic geometry we can make good sense of generic fibres of dominant maps.

Our original motivation is the Poisson Dixmier-Moeglin equivalence (PDME) problem, some aspects of which remain open.
See, for example,~\cite{BrownGordon, BLLSM, LLS}.
In our efforts to bring to bear more ideas from differential-algebraic geometry and model theory to this problem, it became clear that fundamental constructions, like base extension and generic fibres, were necessary, and that these in turn required us to move out of the strictly Poisson setting.

From the model theoretic point of view, a central difficulty is that the class of existentially closed bidifferential fields in characteristic zero is not elementary.
This is pointed out in~$\S$\ref{subsect-nomc} below.
So while our intuition from the model theory of differentially closed fields can still serve as a general guide, all the work has to be done algebraically without recourse to the notions and techniques of geometric stability theory.

Bidifferential algebra has certainly been considered before, but rarely for commutative rings.
The preoccupations of the subject as it exists in the literature \cite{bresar2,bresar1,eremita} are rather different than our own, focusing more on classifying and describing biderivations on noncommutative (semi)prime rings, and the results there tend to trivialise in the commutative case.

Here is a plan of the paper.
In~$\S$\ref{sect-bidiffalg} we develop the foundations of the subject.
This includes establishing the basic extension results (to localisations in Theorem~\ref{localisation}, to algebraic extensions in Theorem~\ref{separable}, and to transcendental extension in Theorem~\ref{transcendental}).
It also includes the construction of a biderivation on certain tensor products motivated by the generic fibre problem discussed above -- this is Theorem~\ref{tensor} below.
Then in~$\S$\ref{sect-balggeom} we consider the geometric counterpart of the algebra developed earlier.
In particular we introduce $B$-varieties and show that dominant $B$-morphisms admit generic $B$-fibres.
Finally, in~$\S$\ref{sect-bdme} we articulate the Dixmier-Moeglin Equivalence problem in this setting, as a generalisation of the PDME, and establish some preliminary results (Propositions~\ref{usualimplications}, \ref{fibreimage-lc}, \ref{fibreimage-r}) whose proofs illustrate how commutative bidifferential algebra and the constructions appearing in earlier sections can contribute.

\medskip
\noindent
{\em Rings throughout are unital and commutative.}

\bigskip
\section{Bidifferential algebra}
\label{sect-bidiffalg}

\noindent
A {\em bidifferential ring} is a pair $(R,\bider)$ where $R$ is a (commutative unitary) ring and $\bider:R\times R\to R$ is a function such that for each $r\in R$ both $\{r,\cdot\}$ and $\{\cdot,r\}$, called the associated {\em hamiltonians}, are derivations on $R$.
We call $\bider$ a {\em biderivation}.

Examples abound.
Every Poisson bracket is a biderivation, and these are the main motivating examples.
Another source of examples comes from taking linear combinations of products of derivations: If $\{\delta_1,\dots,\delta_m\}$ and $\{\partial_1,\dots,\partial_m\}$ are derivations on a ring $R$, and $\alpha_1,\dots,\alpha_m\in R$, then
$$\{r,s\}:=\alpha_1\delta_1(r)\;\partial_1(s)+\cdots + \alpha_m\delta_m(r)\;\partial_m(s)$$
defines a biderivation on $R$.

Suppose $(R,\bider)$ is a bidifferential ring.
A {\em bidifferential ideal} of $(R,\bider)$ is an ideal $I\subseteq R$ satisfying $\{R,I\}\subseteq I$ and $\{I,R\}\subseteq I$ -- that is, an ideal that is preserved by all the associated hamiltonians.
A {\em morphism} of bidifferential rings, $\phi:(R,\bider)\to (S,\bider)$, is a ring homomorphism such that $\phi\{a,b\}=\{\phi a,\phi b\}$ for all $a,b\in R$.
The reader can easily verify for themself the first properties of this category; for example, that the kernel of a morphism of bidifferential rings is a bidifferential ideal, and that the quotient of a bidifferential ring by a bidifferential ideal admits a unique biderivation making the quotient map a morphism of bidifferential rings.
But in verifying further properties, especially having to do with extensions of bidifferential rings, it becomes convenient to work in the following more general (if technical) setting.

\bigskip
\subsection{Multisorted biderivations}
We extend the notion of biderivation so that the arguments accept values from different rings and output values in a module.
This will allow us to treat all three sorts individually.

Let $\K$ denote the collection of triples $(R,S,M)$ where $R$ and $S$ are rings and $M$ is an additive group equipped with both an $R$-module and $S$-module structure.
A {\em biderivation on $(R,S,M)$} is then a function $\{\cdot,\cdot\}:R\times S\to M$ such that
\begin{itemize}
\item
$\{r,\cdot\}:S\to M$ is an $M$-valued derivation on $S$, for each $r\in R$, and
\item
$\{\cdot,s\}:R\to M$ is an $M$-valued derivation on $R$, for each $s\in S$.
\end{itemize}
We call $\{r,\cdot\}$ and $\{\cdot,s\}$ the associated {\em hamiltonians}.
We often say $(R,S,M,\{\cdot,\cdot\})$ is a biderivation to mean that $(R,S,M)\in\mathcal K$ and $\{\cdot,\cdot\}$ is a biderivation on $(R,S,M)$.

We recover (one-sorted) bidifferential rings as biderivations on a triple of the form $(R,R,R)$, where both module structures (namely of the third sort over the first and over the second) are obtained by viewing $R$ as a module over itself in the natural way.
In this case we write $(R,\bider)$ instead of $(R,R,R,\bider)$.

An intermediate level of generalisation that arises frequently is when $R=S$ and the two $R$-module structures on $M$ agree.
In that case we say that $\bider$ is an {\em $M$-valued  biderivation on $R$}.
We may write this as $(R,\bider)$ as well, but only when it is clear from context in which module the biderivation takes its values.

By a {\em morphism in $\mathcal K$}, denoted $\phi:(R,S,M)\to(R',S',M')$, we mean a triple $(\phi_1,\phi_2,\phi_3)$ such that
\begin{itemize}
\item
$\phi_1:R\to R'$ and $\phi_2:S\to S'$ are ring homomorphisms, and
\item
$\phi_3:M\to M'$ is a group homomorphism that is both $R$-linear and $S$-linear, in the sense that, for all $x\in M$,
\begin{itemize}
\item
$\phi_3(rx)=\phi_1(r)\phi_3(x)$ for all $r\in R$ and
\item
$\phi_3(sx)=\phi_2(s)\phi_3(x)$ for all $s\in S$.
\end{itemize}
\end{itemize}
Given biderivations on  $(R,S,M)$ and $(R',S',M')$ in $\mathcal K$, we say that a morphism $\phi:(R,S,M)\to(R',S',M')$ is {\em bidifferential} if
$$\phi_3(\{r,s\})=\{\phi_1r,\phi_2s\}$$
for all $r\in R$ and $s\in S$.
We denote this by $\phi:(R,S,M,\{\cdot,\cdot\})\to(R',S',M',\{\cdot,\cdot\})$, and say that the biderivation on $(R',S',M')$ {\em extends} that on $(R,S,M)$.
Note that, whenever we can get away with it, we will use the same symbol to denote the base biderivation and its extension.

The following is the main extension lemma for biderivations.

\begin{proposition}
\label{usederivations2}
Suppose
\begin{itemize}
\item
$\phi:(R,S,M)\to (R,S',M')$ is a morphism in $\K$ with $\phi_1=\id_R$,
\item
$\bider$ is a biderivation on $(R,S,M)$,
\item
$a\in S'$, and
\item
$\delta:R\to M'$ is an $M'$-valued derivation on $R$.
\end{itemize}
Suppose each hamiltonian $\{r,\cdot\}:S\to M$ extends to a unique $M'$-valued derivation on $S'$ taking $a$ to $\delta(r)$.
Then $\bider$ extends to a unique biderivation on $(R,S',M')$ satisfying $\{\cdot,a\}=\delta$.
\end{proposition}

\begin{proof}
For each $r\in R$, let $\delta_r:S'\to M'$ denote the (unique) derivation extending $\{r,\cdot\}:S\to M$ and taking $a$ to $\delta(r)$.
Then we have no choice but to define $\bider$ on $(R,S',M')$ by $\{r,s'\}:=\delta_r(s')$.
In particular, uniqueness will follow from existence.
That this does define an extension of $(R,S,M,\bider)$ follows immediately from the fact that each $\delta_r$ extends $\{r,\cdot\}:S\to M$.
So it remains only to show that $(R,S',M',\bider)$ is a biderivation.

Given $r\in R$, we defined $\{r,\cdot\}:S'\to M'$ to agree with $\delta_r$, and hence these are $M'$-valued derivations on $S'$ by assumption.

Given $s'\in S'$, we need to show that $\{\cdot,s'\}:R\to M'$ is a derivation.
That is, for all $r_1,r_2\in R$ we need to show:
\begin{eqnarray*}
\{r_1+r_2,s'\}&=&\{r_1,s'\}+\{r_2,s'\},\ \ \text{ and}\\
\{r_1r_2,s'\}&=&r_2\{r_1,s'\}+r_1\{r_2,s'\}.
\end{eqnarray*}
Fixing $r_1,r_2\in R$ and letting $s'\in S$ vary, this is equivalent to showing that:
\begin{eqnarray}
\label{add}\delta_{r_1+r_2}&=&\delta_{r_1}+\delta_{r_2},\ \ \text{ and}\\
\label{leib}\delta_{r_1r_2}&=&r_2\delta_{r_1}+r_1\delta_{r_2},
\end{eqnarray}
which is what we now do.

For~(\ref{add}), note that both $\delta_{r_1+r_2}$ and $\delta_{r_1}+\delta_{r_2}$ are $M'$-valued derivations on $S'$ that extends $\{r_1+r_2,\cdot\}:S\to M$.
Moreover,
$$\delta_{r_1+r_2}(a)=\delta(r_1+r_2)=\delta(r_1)+\delta(r_2)=\delta_{r_1}(a)+\delta_{r_2}(a)=(\delta_{r_1}+\delta_{r_2})(a).$$
So, by the uniqueness assumption, we must have $\delta_{r_1+r_2}=\delta_{r_1}+\delta_{r_2}$ on $S'$.

For~(\ref{leib}), note that both $\delta_{r_1r_2}$ and $r_1\delta_{r_2}+r_2\delta_{r_1}$ are $M'$-valued derivations on $S'$ that extend $\{r_1r_2,\cdot\}:S\to M$.
Moreover,
$$\delta_{r_1r_2}(a)=\delta(r_1r_2)=r_1\delta(r_2)+r_2\delta(r_1)=r_1\delta_{r_2}(a)+r_2\delta_{r_1}(a)=(r_1\delta_{r_2}+r_2\delta_{r_1})(a).$$
Hence $\delta_{r_1r_2}=r_1\delta_{r_2}+r_2\delta_{r_1}$ follows by uniqueness.
\end{proof}

\begin{remark}
The following observations are worth keeping in mind:
\begin{itemize}
\item[(a)]
Proposition~\ref{usederivations2} is often applied with $a=0$ and $\delta=0$, in which case it simply says that if all the hamiltonians extend uniquely from $S\to M$ to $S'\to M'$ then $\bider$ extends uniquely from $(R,S,M)$ to $(R,S',M')$.
\item[(b)]
Proposition~\ref{usederivations2} is often applied when $S'$ does not admit any nontrivial derivations that are $S$-linear -- the so called {\em $S$-differentially trivial} case.
In that case, one has automatically that each hamiltonian has at most one extension to $S'$ (as the difference of two would be $S$-linear), and the only thing to verify before application is existence.
\item[(c)]
Of course, the proposition also works with the roles of $R$ and $S$ reversed -- that is, we can keep $S$ fixed and extend $R$ to some~$R'$.
\end{itemize}
\end{remark}

\medskip
The multi-sorted generalisation of a bidifferential ideal is the following:

\begin{definition}
A {\em bidifferential ideal pair} of a biderivation $(R,S,M,\bider)$ is a pair of ideals $I\subseteq R$ and $J\subseteq S$ such that $\{I,S\}\subseteq IM$ and $\{R,J\}\subseteq JM$.
\end{definition}

One of the complications in bidifferential algebra is that the extension of a bidifferential ideal need not be bidifferential.
For example:

\begin{example}
\label{nonliftex}
Fix a ring $R$.
It is not hard to construct a biderivation on the polynomial ring in two variables, $R[x,y]$, with the property that $\bider$ is trivial on $R[x]$ but such that $\{x,y\}=x$.
For example, one can use Corollary~\ref{polyringext} below to verify this.
In fact, when $R=k$ is a field, one can do it so that $\bider$ is a Poisson bracket over $k$.
In any case, we thus have an extension of bidifferential rings; $(R[x],0)\to(R[x,y],\bider)$.
The ideal $(x-1)R[x]$ is bidifferential in $(R[x],0)$ as every ideal is, but, since
$\{x-1,y\}=\{x,y\}-\{1,y\}=x$,
 we see that $(x-1)R[x,y]$ is not a bidifferential ideal in the extension $(R[x,y],\bider)$.
\end{example}

We will thus have to restrict our attention to certain extensions of bidifferential rings.
Here is the definition in the general multi-sorted setting:

\begin{definition}[Lifting extensions]
An extension of biderivations,
$$\phi:(R,S,M,\bider)\to(R',S',M',\bider),$$
is said to be {\em lifting} if whenever $(I,J)$ is a bidifferential ideal pair in $(R,S,M)$ then $(IR',JS')$ is a bidifferential ideal pair in $(R',S',M')$.
\end{definition}

We are using the usual commutative algebra conventions here, namely $IR'$ denotes the ideal $\phi_1(I)R'$ of $R'$ generated by the image of $I$ in $R'$, and similarly for $JS':=\phi_2(J)S'$.

The following is a useful criterion for extensions to be lifting.

\begin{proposition}
\label{liftcrit}
Suppose $\phi:(R,S,M,\bider)\to (R,S',M',\bider)$ is a differential morphism with $\phi_1=\id_R$.
If $S'$ differentially trivial\footnote{That is, $S'$ admits no nontrivial $S$-linear derivations -- see Definition~\ref{dtriv}}
 over $S$ then $\phi$ is lifting.
\end{proposition}

\begin{proof}
Fix a bidifferential ideal pair $(I,J)$ in $(R,S,M)$.
We need to show that $(I,JS')$ is bidifferential in $(R,S',M')$.

Let $JM':=(\phi_2(J)S')\cdot M'$, for short.
We first show that $\{R,JS'\}\subseteq JM'$.
By additivity, it suffices to verify that $\{r, xs'\}\in JM'$ for all $r\in R, x\in J$, and $s'\in S'$.
This follows from the following computation:
\begin{eqnarray*}
\{r, xs'\}
&=&
\{r, \phi_2(x)s'\}\\
&=&
\phi_2(x)\{r,s'\}+s'\{r,\phi_2(x)\}\\
&=&
\phi_2(x)\{r,s'\}+s'\phi_3\{r,x\}.
\end{eqnarray*}
Now, $\phi_2(x)\{r,s'\}$ is visibly in $JM'$.
For the second term, note that, as $(I,J)$ is a bidifferential pair, $\{r,x\}\in JM$, and hence $\phi_3\{r,x\}\in JM'$.
So $\{r,xs'\}\in JM'$.

It remains to prove that $\{I,S'\}\subseteq IM'$.
This is where differential triviality is used.
Fix $x\in I$ and consider the derivation $D:=\pi\circ\{x,\cdot\}:S'\to M'/IM'$.
Here $\pi:M'\to M'/IM'$ is the quotient map.
We check that $D$ is $S$-linear.
Indeed, for any $s\in S$, we have
\begin{eqnarray*}
D(\phi_2(s))
&=&
\pi\{x,\phi_2(s)\}\\
&=&
\pi\phi_3\{x,s\}\\
&=&
0\ \ \ \ \text{ as $\{x,s\}\in \{I,S\}\subseteq IM$.}
\end{eqnarray*}
So by differential triviality, $D=0$.
But this says precisely that $\{x,S'\}\subseteq IM'$.
As $x\in I$ was arbitrary, we have $\{I,S'\}\subseteq IM'$.
\end{proof}

\bigskip
\subsection{Extending bidifferential rings}
We can now prove the basic results about extensions of bidifferential rings.
We start with observing that differential triviality ensures uniqueness of extensions.

\begin{lemma}
\label{mdt}
Suppose $R$ is an $A$-algebra and $M$ is an $R$-module.
If $R$ admits no nontrivial $M$-valued $A$-linear derivations then every $M$-valued biderivation on $A$ has at most one extension to $R$.
\end{lemma}

\begin{proof}
Suppose, toward a contradiction, that $(R,R,M,\bider_1)$ and $(R,R,M,\bider_2)$ are distinct extensions of $(A,A,M,\bider)$.
Then there exists $r\in R$ such that the hamiltonians $\delta_1:=\{r,\cdot\}_1$ and $\delta_2:=\{r,\cdot\}_2$ are distinct $M$-valued derivations on $R$.
If $\delta_1$ and $\delta_2$ agreed on $A$ then $\delta_1-\delta_2$ would be an $A$-linear $M$-valued derivation on $R$, and differential triviality would force $\delta_1=\delta_2$.
Hence, there is $a\in A$ such that $\delta_1(a)\neq\delta_2(a)$.
But this means that $d_1:=\{\cdot,a\}_1$ and $d_2:=\{\cdot,a\}_2$ are distinct $M$-valued derivations on $R$.
But $d_1$ and $d_2$ do agree on $A$ because $\bider_1$ and $\bider_2$ both extend $\bider$ and $a\in A$.
So $d_1-d_2$ is a nontrivial $A$-linear $M$-valued derivation on $R$, contradicting our assumption.
\end{proof}

\begin{theorem}[Localisation]
\label{localisation}
Suppose $\bider$ is an $M$-valued biderivation on $R$, and $\Sigma\subseteq R\setminus\{0\}$ is a multiplicatively closed set.
Then there is a unique extension of $\bider$ to an $\Sigma^{-1}M$-valued biderivation on $\Sigma^{-1}R$.
Moreover, this is a lifting extension.
\end{theorem}

\begin{proof}
By the quotient rule, $\Sigma^{-1}R$ is differentially trivial over $R$, and hence uniqueness will follow automatically by Lemma~\ref{mdt}.
On the other hand, every $M$-valued derivation on $R$ does extend to a $\Sigma^{-1}M$-valued derivation on $\Sigma^{-1}R$.
So we can apply Proposition~\ref{usederivations2} with $(a=0,\delta=0)$ twice, once to extend from $(R,R,M,\bider)$ to $(R,\Sigma^{-1}R,\Sigma^{-1}M)$, and then to extend to $(\Sigma^{-1}R,\Sigma^{-1}R,\Sigma^{-1}M)$.
This is lifting because, by differential triviality, Proposition~\ref{liftcrit} applies to both.
\end{proof}

\begin{corollary}
\label{fractionext}
A biderivation on an integral domain extends uniquely to the fraction field.
\end{corollary}

\begin{proof}
Apply Theorem~\ref{localisation} to $M=R$ and $\Sigma=R\setminus\{0\}$.
\end{proof}

Next we wish to show that we can lift biderivations uniquely to (separably) algebraic extensions.
We articulate a rather more general, though somewhat technical, version that allows the base of an algebra to be extended in this way.
The greater level of generality is in fact necessary for what we do in the next section.

\begin{theorem}[Algebraic base extensions]
\label{separable}
Suppose $A\subseteq B$ is an extension of integral domains, $b\in B$ is separably algebraic over $\Frac(A)$, and $\bider$ is a biderivation on an $A$-algebra~$R$ with values in $R\otimes_AB$.
Then there is a nonzero $f\in A[b]$ such that $\bider$ extends uniquely to a biderivation on $R\otimes_AA[b]_f$ with values in~$R\otimes_AB_f$.
Moreover, this is a lifting extension.
\end{theorem}

\begin{proof}
Note that in the statement the $R$-module structure on $R\otimes_AB$, and the $(R\otimes_AA[b])$-module structure on $R\otimes_AB_f$, are the natural ones.

The first thing we do is reduce to the case when the minimal polynomial of $b$ over $\Frac(A)$ has coefficients in $A$.
Indeed, let $P\in A[t]$ be a nonzero polynomial of minimal degree such that $P(b)=0$, and let $a$ be the leading coefficient of~$P$.
By Theorem~\ref{localisation} we have a unique extension of $\bider$ to the localisation $R_a=R\otimes_AA_a$ with values in $(R\otimes_AB)_a=R_a\otimes_{A_a}B_a$.
Now, the minimal polynomial of $b$ over $\Frac(A_a)=\Frac(A)$, namely $\frac{1}{a}P$, does have coefficients in $A_a$.
So, assuming we have proved the proposition in this case, we would get a nonzero $g\in A_a[b]$ and a unique extension of $\bider$ to
$R_a\otimes_{A_a}A_a[b]_g=R\otimes_AA[b]_{ag}$
with values in $R_a\otimes_{A_a}B_{ag}=R\otimes_AB_{ag}$.
Moreover this extension would be a lifting extension of the original $\bider$.
Hence $f:=ag$ would witness the truth of the proposition.

So we may assume that the minimal polynomial of $b$ over $\Frac(A)$ has coefficients in $A$.
We denote this minimal polynomial by $P$ and set $f:=\frac{dP}{dt}(b)$.
By separability, $f$ is a nonzero element of $A[b]$.
It is pointed out in the appendix, as Fact~\ref{alg-derivation}, that under these conditions every derivation $d:A\to S$ (where $S$ is any $B$-algebra) extends uniquely to $d:A[b]\to S_f$.
Combining this with the fact that one can naturally take tensor products of derivations (see Fact~\ref{tensor-derivation}), we get that every derivation $d:R\to R\otimes_AB$ extends uniquely to a derivation $d:R\otimes_AA[b]\to R\otimes_AB_f$.
Now, applying this to each each hamiltonian $\{r,\cdot\}$, viewed as $R\otimes_AB$-valued, we see that Proposition~\ref{usederivations2} applies (with $a=0$ and $\delta=0$).
We obtain an extension of $\bider$ to $(R,R\otimes_AA[b],R\otimes_AB_f)$.
Next, applying the same fact to the hamiltonians $\{\cdot,s\}$ for each $s\in R\otimes_AA[b]$, we can invoke Proposition~\ref{usederivations2} again, this time to the first argument, so as to obtain a further extension to $(R\otimes_AA[b],R\otimes_AA[b],R\otimes_AB_f)$.
It then extends to $(R\otimes_AA[b]_f,R\otimes_AA[b]_f,R\otimes_AB_f)$ by Theorem~\ref{localisation}.

Note that if $d$ is an $A$-linear derivation on $A[b]_f$, with values in any $A[b]_f$-module~$M$, then $P(b)=0$ forces $0=\frac{dP}{dt}(b)db=fdb$, and hence $db=0$ by multiplying through by $\frac{1}{f}$.
So $A[b]_f$ is differentially trivial over $A$, and it follows that $R\otimes_AA[b]_f$ is differentially trivial over $R$.
By Lemma~\ref{mdt}, the extension we have built is therefore unique.
Also, applying Proposition~\ref{liftcrit} to
$$(R,R,R\otimes_AB,\bider)\to (R,R\otimes_AA[b]_f,R\otimes_AB_f,\bider)$$
and then to
$$(R,R\otimes_AA[b]_f,R\otimes_AB_f,\bider)\to(R\otimes_AA[b]_f,R\otimes_AA[b]_f,R\otimes_AB_f,\bider)$$
shows that the extension we have built is lifting.
\end{proof}

\begin{remark}
\label{whatf}
The proposition will only be used when $R\otimes_AB$ has no $A$-torsion.
It is not hard to show that in that case, in order to extend the hamiltonions to $R\otimes_AA[b]$, one does not require that the minimal polynomial of~$b$ over $\Frac(A)$ has coefficients in $A$.
(See the proof of Fact~\ref{alg-derivation}.)
Hence, assuming that $R\otimes_AB$ has no $A$-torsion, we can, in Theorem~\ref{separable},  take $f$ to be $\frac{dP}{dt}(b)$ where $P\in A[t]$ is a polynomial of minimal degree such that $P(b)=0$.
\end{remark}

\begin{corollary}
\label{sepcl}
A biderivation on a field extends uniquely to the separable closure.
\end{corollary}

\begin{proof}
This can be deduced from recursive applications of Theorem~\ref{separable}, starting with $A=R=:k$ a field, and $B=k^{\sep}$, and extending the biderivation one element at a time.
Or, and this may be easier, we can directly apply Proposition~\ref{usederivations2} to first extend $\bider$ uniquely to $(k,k^{\sep},k^{\sep})$, and then apply Proposition~\ref{usederivations2} again to extending the first argument to $k^{\sep}$.
Both stages use the fact that derivations $d:k\to k^{sep}$ extend uniquely to a derivation on $k^{\sep}$.
\end{proof}

\begin{theorem}[Transcendental base extensions]
\label{transcendental}
Suppose $A\subseteq B$ is an extension of rings and $b\in B$ is transcendental over $A$.
Suppose $R$ is an $A$-algebra, $D:R\to R\otimes_AB$ is a derivation, and $E:R\otimes_AA[b]\to R\otimes_AB$ is a derivation.
Every $(R\otimes_AB)$-valued biderivation on $R$ extends uniquely to a $(R\otimes_AB)$-valued biderivation on $R\otimes_AA[b]$ satisfying
\begin{itemize}
\item
$\{r\otimes 1,1\otimes b\}=D(r)$ for all $r\in R$, and
\item
$\{1\otimes b,x\}=E(x)$ for all $x\in R\otimes_AA[b]$.
\end{itemize}
\end{theorem}

\begin{proof}
First apply Proposition~\ref{usederivations2} with $S=R$, $M=M'=R\otimes_AB$, $S'=R\otimes_AA[b]$, and with $a:=1\otimes b$ and $\delta:=D$.
Note that $R\otimes_AA[b]=R[a]$ is isomorphic to the polynomial ring in one variable over~$R$, and so for each $r\in R$ it is true that $\{r,\cdot\}:R\to R\otimes_AB$ does have a unique extension to a derivation on $R[a]$ taking $a$ to~$\delta(r)$.
Hence, we get a unique extension $(R,R\otimes_AA[b], R\otimes_AB,\bider)$.
Now we apply Proposition~\ref{usederivations2} again, this time to extend the first argument from $R$ to $R\otimes_AA[b]$.
For that, we use again $a:=1\otimes b$ and this time $\delta:=E$.
For each $x\in R\otimes_AA[b]$, the hamiltonian $\{\cdot,x\}:R\to R\otimes_AB$ has a unique extension to a derivation on $R[a]=R\otimes_AA[b]$ taking $a$ to~$\delta(x)$.
The extended $(R\otimes_AB)$-valued biderivation on $R\otimes_AA[b]$  that we obtain has by construction the two desired properties, and is by construction unique such.
\end{proof}

\begin{corollary}
\label{polyringext}
Every biderivation on $R$ extends to the polynomial ring~$R[t]$.
Moreover, given a pair of derivation $D:R\to R[t]$ and $E:R[t]\to R[t]$, there is a unique extension of $\bider$ to $R[t]$ satisfying $\{\cdot,t\}\upharpoonright_R=D$ and $\{t,\cdot\}=E$.
\end{corollary}

\begin{proof}
Apply Theorem~\ref{transcendental} with $A=R$, $B=R[t]$, $b=t$, and $D, E$ as given.
\end{proof}

\bigskip
\subsection{A model-theoretic aside}
\label{subsect-nomc}
We can view bidifferential rings as first-order structures in the language $L$ of rings augmented by a binary function symbol for the biderivation.
But there is no good model theory of biderivations in the sense that the class of existentially closed bidifferential fields is not elementary.
Let us show how this follows from some of the basic extension results proved above.

\begin{proposition}
The $L$-theory of bidifferential fields of characteristic zero does not admit a model companion.
\end{proposition}

\begin{proof}
First of all, any biderivation extends (uniquely) to the separable closure by Corollary~\ref{sepcl}.
It follows that an existentially closed bidifferential field of characteristic zero is algebraically closed.

By the {\em constants} of a bidifferential ring $(R,\bider)$ we mean the subring of elements $r\in R$ such that the associated hamiltonians $\{r,\cdot\}$ and $\{\cdot,r\}$ are identically zero.
We observe that the constant subfield of an existentially closed bidifferential field $(F,\bider)$ is always $\mathbb Q^{\alg}$.
Indeed, consider the function field $F(t)$ in one variable.
Given any derivation $D:F\to F$, we can, using Corollaries~\ref{fractionext} and~\ref{polyringext}, build a biderivation on $F(t)$ extending $\bider$ on $F$ and such that $\{\cdot,t\}\upharpoonright_F=D$.
Suppose, now, that we are given $r\in F\setminus\mathbb Q^{\alg}$.
Then we can find $D$ such that $D(r)\neq 0$.
Hence $\{r,t\}=D(r)\neq 0$.
It follows by existential closedness, that there is some $s\in F$ such that $\{r,s\}\neq 0$.
So $r$ is not a constant in $(F,\bider)$.

Since the constants, which form a $0$-definable infinite set, does not grow within the class of existentially closed bidifferential fields, that class cannot be elementary.
\end{proof}

\bigskip
\subsection{Tensor products}
Unlike in the differential case (see Fact~\ref{tensor-derivation}), biderivations are not always compatible with tensor products.
The following example shows that we cannot always equip the tensor product of bidifferential rings with a biderivation:

\begin{example}
Let $k$ be a field and consider $R:=k[x,y]\otimes_{k[x]} k$, where the polynomial ring $k[x,y]$ is viewed as a $k[x]$-algebra in the natural way, and $k$ is viewed as a $k[x]$-algebra via the evaluation at zero homomorphism.
Now equip $k$ and $k[x]$ with trivial biderivations, while equipping $k[x,y]$ with the biderivation extending the trivial one on $k[x]$ and satisfying $\{\cdot,y\}\upharpoonright_{k[x]}=\frac{d}{dx}$.
One can use Corollary~\ref{polyringext} to see this exists.
We {\bf cannot} equip $R$ with a bidifferential structure extending $(k[x,y],\bider)$.
Indeed, if we could, then using that $1\otimes x=0$ in $R$, we would have
$0=\{y\otimes 1,1 \otimes x\}=\{y\otimes 1,x\otimes 1\}=\{x ,y\}\otimes 1=1\otimes 1$,
which is a contradiction as $R$ is not the trivial ring.
\end{example}

Even if we are interested primarily in tensor products over fields, as we are, problems persist.
Let us fix a bidifferential field $(k,\bider)$ and consider the category of {\em bidifferential $k$-algebras}, namely $k$-algebra equipped with a biderivation that extends $\bider$ on $k$.
We would like to to equip the tensor product $R\otimes_k S$ of bidifferential $k$-algebras with a suitable biderivation.
But, again unlike for differential rings, there is no truly canonical choice.
At first glance, at least when the biderivation is trivial on~$k$, it does seem that a canonical choice would be to define
\begin{equation}
\label{canonical}
\{r_1\otimes s_1,r_2\otimes s_2\}=\{r_1,r_2\}\otimes s_1s_2+r_1r_2\otimes\{s_1,s_2\},
\end{equation}
as is done in Poisson algebra.
Indeed, (\ref{canonical}) induces the unique biderivation on $R\otimes_k S$ that extends the biderivations on $R$ and $S$ and that sets $\{r\otimes 1,1\otimes s\}=0$.
But, as the following example shows, that choice has some pathological implications.

\begin{example}
\label{diagonalex}
Suppose $\bider$ is trivial on $k$ and make $R=k[x,y]$ the bidifferential $k$-algebra where $\{x,y\}=x$.
(So this is the same as Example~\ref{nonliftex}.)
If we equip $R\otimes_kR=k[x\otimes 1,y\otimes 1,1\otimes x,1\otimes y]$, with the canonical biderivation given by~(\ref{canonical}) then the diagonal ideal
$I:=\big(x\otimes 1-1\otimes x,\ \  y\otimes 1-1\otimes y\big)$
is not bidifferential.
Indeed,
$\{x\otimes 1-1\otimes x,y\otimes 1\}=\{x,y\}\otimes 1=x\otimes 1$,
which is not an element of $I$.
\end{example}

In any case, when $\bider$ is not trivial on~$k$, (\ref{canonical}) is simply not available to us: if $r,s\in k$ then $\{r\otimes 1,1\otimes s\}=\{r\otimes 1,s\otimes 1\}$ which should agree with $\{r,s\}\otimes 1$  and hence $\{r\otimes 1,1\otimes s\}$ should {\em not} be zero if $\{r,s\}\neq 0$ in $k$.

We propose an alternative bidifferential structure on certain tensor products of bidifferential $k$-algebras.
First, we will mostly be interested in the affine case:

\begin{definition}
An {\em affine bidifferential $k$-algebra} is a finitely generated $k$-algebra that is an integral domain and is equipped with a biderivation extending~$(k,\bider)$.
\end{definition}

The following statement may seem technical, but it has a clear geometric formulation and motivation that will be expressed later by Theorem~\ref{tensor-geometric}.
For now, let us only point out that $\iota$ corresponds to a dominant morphism of affine algebraic varieties, and the defining ideal of $\iota$ in the tensor product corresponds to the graph of that morphism.

\begin{theorem}
\label{tensor}
Suppose $(k,\bider)$ is a bidifferential field, $(R,\bider)$ and $(S,\bider)$ are affine bidifferential $k$-algebras with $\Frac(S)$ separable over $k$.
Suppose $\iota:S\to R$ is a $k$-linear bidifferential embedding.
Then there exists a nonzero $f\in S$ and a biderivation on $R\otimes_k S_f$ satisfying the following properties:
\begin{itemize}
\item[(a)]
$(R\otimes_kS_f,\bider)$ is a lifting extension of $(R,\bider)$.
\item[(b)]
$(R\otimes_kS_f,\bider)$ is an extension of $(S,\bider)$.
\item[(c)]
the subring $\iota(S)\otimes_kS_f\subseteq R\otimes_kS_f$ is bidifferential, and it is a lifting extension of $(\iota(S),\bider)$.
\item[(d)]
The defining ideal of $\iota$, namely the ideal generated by $\iota s\otimes 1-1\otimes s$ for all $s\in S$,
is a bidifferential ideal of $(R\otimes_kS_f,\bider)$.
\end{itemize}
\end{theorem}

\begin{remark}
In particular, if we take $S=R$ and $\iota=\id$ then this gives us a biderivation on (a localisation) of $R\otimes_kR$ avoiding the pathologies exhibited by Example~\ref{diagonalex}.
\end{remark}

\begin{proof}[Proof of Theorem~\ref{tensor}]
By the separable form of Noether's normalisation lemma (see~\cite[16.18]{eisenbudbook}) there is $y:=(y_1,\dots,y_n)$ in $S$ algebraically independent over $k$ such that $\Frac(S)$ is separably algebraic over $k(y)$.

{\em Step~1 of the construction.}
We extend $\bider$ from $R$ to an $R\otimes_kS$-valued biderivation on $R\otimes_kk[y]$ using Theorem~\ref{transcendental} recursively.
For this, we view the biderivation on $R$ as being $R\otimes_kS$-valued, and the biderivation on $S$ restricting to an $R\otimes_kS$-valued biderivation on $k[y]$.
That is, we have
\begin{eqnarray*}
\bider&:&R\times R\to R\otimes_kS\\
\bider&:&k[y]\times k[y]\to R\otimes_kS
\end{eqnarray*}
Suppose we have extended these to $R\otimes_kk[y_1,\dots,y_{i-1}]$.
Define the derivation $D:R\otimes_kk[y_1,\dots,y_{i-1}]\to R\otimes_kS$ to be the derivation determined by
\begin{eqnarray*}
D(r\otimes 1)&=&\{r,\iota y_i\}\otimes 1\ \ \text{for all $r\in R$, and}\\
D(1\otimes y_j)&=&1\otimes\{y_j,y_i\}\ \ \text{for all $j< i$,}
\end{eqnarray*}
and define 
$E:R\otimes_kk[y_1,\dots,y_i]\to R\otimes_kS$ to be the derivation determined by:
\begin{eqnarray*}
E(r\otimes 1)&=&\{\iota y_i,r\}\otimes 1\ \ \text{for all $r\in R$, and}\\
E(1\otimes y_j)&=&1\otimes\{y_i,y_j\}\ \ \text{for all $j\leq i$.}
\end{eqnarray*}
Now apply Theorem~\ref{transcendental} with this $D$ and $E$ to obtain $\bider$ on $R\otimes_kk[y_1,\dots,y_i]$.
Recursively, we obtain a unique $R\otimes_kS$-valued biderivation on $R\otimes_kk[y]$ that  extends the given biderivations on $R$ and $k[y]$ and satisfies:
\begin{eqnarray}
\label{choicer}\{r\otimes 1, 1\otimes y_i\}&=&\{r,\iota y_i\}\otimes 1\\
\label{choicel}\{1\otimes y_i,r\otimes 1\}&=&\{\iota y_i,r\}\otimes 1
\end{eqnarray}
for all $r\in R$ and $i=1,\dots,n$.

{\em Step 2 of the construction.}
We now extend $\bider$ from $R\otimes_kk[y]$ to $R\otimes_kS$, at the expense of localising~$S$.
As $S$ is a finitely generated and $\Frac(S)$ is separably algebraic over $k(y)$,  we can write $S=k[y,b]$ where $b=(b_1,\dots,b_m)$ are separably algebraic over $k(y)$.
For each $i$ let $P_i\in k[y,b_1,\dots,b_{i-1}][t]$ be of minimal degree such that $P_i(b_i)=0$.
Let $\displaystyle f:=\prod_{i=1}^m\frac{dP_i}{dt}(b_i)$.
Applying Theorem~\ref{separable} repeatedly we obtain an extension of $\big(R\otimes_kk[y],\bider\big)$ to $\big(R\otimes_kS_f,\bider\big)$.
Actually, we are using here that $R\otimes_kS$ has no $S$-torsion to extract from Theorem~\ref{separable} that this particular $f$ works -- otherwise we only know there is some $f$ which does. See Remark~\ref{whatf}.

We now verify~(a) through~(d).

$(a).$
By construction, $(R,\bider)\to \big(R\otimes_kk[y],\bider\big)\to \big(R\otimes_kS_f,\bider\big)$ are extensions.
Here we view all the above biderivations as $R\otimes_kS_f$-valued under the natural identifications.
We verify that both steps are lifting extensions.
The second is lifting because it is obtained by repeated applications of Theorem~\ref{separable} and then an application of Theorem~\ref{localisation}, and both of these theorems tell us that the extensions they produce are lifting.
To verify that the first extension is lifting, fix a bidifferential ideal pair $(I,J)$ in $(R,\bider)$.
We first show that
$$\{I(R\otimes_kk[y]),R\otimes_kk[y]\}\subseteq I(R\otimes_kS).$$
Let $x\in I$.
It suffices to show that $\{x\otimes 1,r\otimes1\}$ and $\{x\otimes 1,1\otimes y_i\}$ are in $I(R\otimes_kS)$ for all $r\in R$ and $i=1,\dots,n$.
The former is because 
$\{x\otimes 1,r\otimes1\}=\{x,r\}\otimes 1$ and  $\{I,R\}\subseteq I$.
The second is because
$\{x\otimes 1,1\otimes y_i\}=\{x,\iota y_i\}\otimes 1$ by~(\ref{choicer}) and again using that $\{I,R\}\subseteq I$.
A similar argument, using~(\ref{choicel}) and the fact that $\{R,J\}\subseteq J$, shows that $\{R\otimes_kk[y],J(R\otimes_kk[y])\}\subseteq J(R\otimes_kS)$.
So $(I(R\otimes_kk[y]),J(R\otimes_kk[y]))$ is a bidifferential ideal pair, as desired.

$(b).$
Let $\bider'$ be the $(R\otimes_kS_f)$-valued biderivation on $S$ obtained by restricting $\big(R\otimes_kS_f,\bider\big)$ to $S=1\otimes S$.
We want to show that $\bider'=\bider$.
Note that $\bider'$ and $\bider$ agree on $k[y]$ because 
by construction,
$$(k[y],\bider)\to \big(R\otimes_kk[y],\bider\big)\to \big(R\otimes_kS_f,\bider\big)$$
are extensions.
Now, even though $\Frac(S)$ is separably algebraic over $k(y)$, it is not exactly the case that $S$ is differentially trivial over $k[y]$.
The problem arises when one consider derivations on $S$ with values in an $S$-module that has torsion.
However, it is true that if $M$ is a torsion-free $S$-module then $S$ has no nontrivial $k[y]$-linear $M$-valued derivations.
It follows by Lemma~\ref{mdt} that there is at most one  $M$-valued biderivation on $S$ extending any given $M$-valued biderivation on $k[y]$.
Applying this to $M:=R\otimes_kS_f$, which has no $S$-torsion as it is free over $S_f$, 
we get that $\bider'=\bider$ on $S$, as desired.

$(c).$
Since $\iota:S\to R$ is bidifferential and $k$-linear, $\iota(S)$ is a bidifferential $k$-subalgebra of $R$.
Inspecting the construction of $\bider$ on $R\otimes_kS_f$ readily reveals that $\iota(S)\otimes_kS_f$ is a bidifferential subring of $R\otimes_kS_f$.
In fact, the restriction of $\bider$ to $\iota(S)\otimes_k S_f$ is just the biderivation our construction yields if we had viewed $\iota$ as $\iota(S)$-valued rather than as $R$-valued.
So part~(a) of this theorem applied to the case when $R=\iota(S)$, gives us that $\iota(S)\otimes_kS_f$ is a lifting extension of $\iota(S)$.

$(d).$
Let $d:S\to R\otimes_kS_f$ be the function $d(s):=\iota s\otimes1-1\otimes s$, and let $I$ be the ideal of $R\otimes_kS_f$ generated by $\{d(s):s\in S\}$.
We wish to show that $I$ is a bidifferential ideal.
We begin with the following approximation:
\begin{claim}
\label{i0}
Let $I_0$ be the ideal in $R\otimes_kk[y]$ generated by $\{d(s):s\in k[y]\}$.
Then $I_0(R\otimes_kS_f)$ is a bidifferential ideal.
\end{claim}

\begin{claimproof}
By construction, and Theorems~\ref{localisation} and~\ref{separable}, $(R\otimes_KS_f,\bider)$ is a lifting extension of $(R\otimes_kk[y],\bider)$.
Hence it suffices to show that $I_0$ is a bidifferential ideal.
We have to show that $\{x,d(s)\}$ and $\{d(s),x\}$ are in $I_0$ for all $x\in R\otimes_kk[y]$ and $s\in k[y]$.
By symmetry, it will suffice to show the former.
It is not hard to see that it suffices to prove this when $x=r\otimes y_i$ and $s=y_j$, for some $r\in R$ and $i,j\leq m$.
This is what we do.

We compute
\begin{align*}
\{r\otimes y_i,1\otimes y_j\} & = \{r\otimes 1,1\otimes y_j\} 1\otimes y_i+r\otimes 1\{1\otimes y_i,1\otimes y_j\} \\
& = \{r,\iota y_j\}\otimes y_i + r\otimes \{y_i,y_j\}\ \ \ \ \text{ by~(\ref{choicer})}
\end{align*}
and 
\begin{align*}
\{r\otimes y_i,\iota y_j\otimes 1\} & = \{r\otimes 1,\iota y_j\otimes 1\} 1\otimes y_i+r\otimes 1\{1\otimes y_i,\iota y_j\otimes 1\} \\
& = \{r,\iota y_j\}\otimes y_i + r\{\iota y_i,\iota y_j\}\otimes 1 \ \ \ \ \text{ by~(\ref{choicel})}\\
& = \{r,\iota y_j\}\otimes y_i + r\iota(y_i,y_j\})\otimes 1 \ \ \ \ \text{ as $\iota$ is bidifferential.}\\
\end{align*}
Hence, we get
\begin{eqnarray*}
\{r\otimes y_i,d(y_j)\}
&=&
\{r\otimes y_i,\iota y_j\otimes 1\}-\{r\otimes y_i,1\otimes y_j\}\\
&=&
r\otimes 1\left(\iota(\{y_i,y_j\})\otimes 1- 1\otimes  \{y_i,y_j\})\right)\\
&=&
(r\otimes1)d(\{y_i,y_j\})\\
&\in& I_0
\end{eqnarray*}
as desired.
\end{claimproof}

To pass from $I_0$ to $I$ we will use the following derivative-like property of $d$.

\begin{claim}
\label{derivative-like}
Suppose $P\in S[t]$ and $s\in S$.
Then $d(P(s))=ad(s)+P^d(\iota s\otimes1)$ where $a\equiv 1\otimes\frac{dP}{dt}(s)\mod I$.
Here $P^d$ is the polynomial over $R\otimes_kS_f$ obtained from $P$ by applying $d$ to the coefficients.
\end{claim}

\begin{claimproof}
By additivity it suffices to prove this in the case that $P(t)=bt^n$ is a monomial.
First, one verifies readily that, for all $s_1,s_2\in S$,
$$d(s_1s_2)=d(s_1)(\iota s_2\otimes1)+(1\otimes s_1)d(s_2).$$
From this, by induction on $n>0$, it follows that
$$d(s^n)=(\iota s^{n-1}\otimes 1+\iota s^{n-2}\otimes s+\cdots+\iota s\otimes s^{n-2}+1\otimes s^{n-1})d(s)$$
and hence
\begin{eqnarray*}
d(P(s))&=&
d(bs^n)\\
&=&
d(b)(\iota s^n\otimes1)+(1\otimes b)(\iota s^{n-1}\otimes 1+\iota s^{n-2}\otimes s+\cdots+1\otimes s^{n-1})d(s)\\
&=&
P^d(\iota s\otimes1)+ad(s)
\end{eqnarray*}
where $a:=(1\otimes b)(\iota s^{n-1}\otimes 1+\iota s^{n-2}\otimes s+\cdots+1\otimes s^{n-1})$.
Now, note that
$$\iota s^{i-1}\otimes s^{i-1}\equiv 1\otimes s^{n-1}\mod I$$
for each $i>0$.
And so,
\begin{eqnarray*}
a&=&
(1\otimes b)(\iota s^{n-1}\otimes 1+\iota s^{n-2}\otimes s+\cdots+1\otimes s^{n-1})\\
&\equiv&
(1\otimes b)(1\otimes ns^{n-1})\mod I\\
&=&
1\otimes\frac{dP}{dt}(s)\mod I
\end{eqnarray*}
as desired.
\end{claimproof}

Now, to prove that $I$ is bidifferential it suffices to check that $\{x,d(s)\}$ and $\{d(s),x\}$ are in $I$ for all $s\in S$ and $x\in R\otimes_kS_f$.
We will prove the former, the latter follows by symmetry.

Recall that $S$ is separably algebraic over $k(y)$.
Let $P\in k[y][t]$ be of minimal degree such that $P(s)=0$.
Then
\begin{eqnarray*}
0&=&
\{x,d(P(s))\}\\
&=&
\{x, ad(s)+P^d(\iota s\otimes1)\}\ \ \ \text{ by Claim~\ref{derivative-like}}\\
&=&
a\{x, d(s)\}+d(s)\{x, a\}+\{x,P^d(\iota s\otimes1)\}\\
&\equiv&
(1\otimes\frac{dP}{dt}(s))\{x, d(s)\}+\{x,P^d(\iota s\otimes1)\}\mod I.
\end{eqnarray*}
Now, as the coefficients of $P$ are in $k[y]$, we have that $P^d(\iota s\otimes1)\in I_0$.
Hence, by Claim~\ref{i0}, $\{x,P^d(\iota s\otimes1)\}\in I_0(R\otimes_kS_f)\subseteq I$.
It follows that
$$(1\otimes\frac{dP}{dt}(s))\{x, d(s)\}\in I.$$
This is a good time to observe that $I$ is (the localisation of) the kernel of the $k$-linear homomorphism $R\otimes_kS\to R$ given by $r\otimes s\mapsto r\iota(s)$.
In particular, as $\frac{dP}{dt}(s)\neq0$ by minimal choice of degree, and $\iota$ is injective, we have that $(1\otimes\frac{dP}{dt}(s)))\notin I$.
But, as $R$ is an integral domain, $I$ is prime.
Hence $\{x,d(s)\}\in I$, as desired.
\end{proof}

\bigskip
\section{Bidifferential-algebraic geometry}
\label{sect-balggeom}

\noindent
In this section, we fix a bidifferential field $(k,\bider)$ of characteristic zero.
The following is the bidifferential analogue of the $D$-varieties studied in differential-algebraic geometry.

\begin{definition}
A {\em $B$-variety} over $k$ is an affine algebraic variety $X$ over $k$ together with a biderivation on $k[X]$ extending $(k,\bider)$.

A {\em $B$-subvariety} of $X$ is a subvariety whose vanishing ideal is a bidifferential ideal.

A {\em $B$-point} of $X$ over $k$ is a $k$-point which as a singleton is a $B$-subvariety.
The set of all $B$-points is denoted by $X^\sharp(k)$.

A {\em morphism} of $B$-varieties (or a {\em $B$-morphism}) is a regular map $\phi:X\to Y$ such that the induced $k$-algebra homomorphism $\phi^*:k[Y]\to k[X]$ is bidifferential.
We say $\phi$ is a {\em lifting} morphism if $\phi^*$ is a lifting extension.
\end{definition}

Note that if $\phi:X\to Y$ is a lifting morphism then $\phi^{-1}(Z)$ is a $B$-subvariety of~$X$, for every $B$-subvariety $Z\subseteq Y$.
This is not necessarily true if $\phi$ is not lifting.

The irreducible components of a $B$-subvariety are themselves $B$-subvarieties.
Indeed, one applies to each of the hamiltonians of $(k[X],\bider)$ the fact that the minimal prime ideals containing a radical differential ideal in a noetherian differential ring are all differential.
See~\cite[\S1.2]{kolchin73}, for instance.

\begin{remark}
It is tempting to associate to a $B$-variety $X$ the various $D$-variety structures arising from the hamiltonians.
The problem is that there is no reason why, for a given $a\in k[X]$, the derivation $\{a,\cdot\}:k[X]\to k[X]$ should extend a ($k$-valued) derivation on the field $k$, in general it is only $k[X]$-valued.
However, if it {\em happens} that $\{a,\cdot\}$ restricts to a derivation $\delta_a:k\to k$ then one can consider the $D$-variety $(X,s_a)$ over $(k,\delta_a)$ given by the section $s_a:X\to \tau X$ induced by $\{a,\cdot\}$.
(Here $\tau X$ denotes the prolongation of $X$ with respect to $\delta_a$.)
In particular, if $\{a,\cdot\}$ is $k$-linear then $s_a:X\to TX$ puts on $X$ the structure of a $D$-variety over the constants.
For example, if $X$ is a Poisson variety over $k$, and so all the hamiltonians are $k$-linear, we can associate to $X$ the collection $\{s_a:X\to TX:a\in k[X]\}$ of $D$-variety structures on $X$ induced by the hamiltonians, and these $D$-variety structures do to some extent capture the Poisson structure. This connection between Poisson and D-variety structures has been exploited in \cite{BLLSM, BLSM, LLS} to answer questions around the Poisson Dixmier-Moeglin equivalence.
\end{remark}

Given a $B$-variety $X$ over $k$ and a field extension $k\subseteq L$, it will of course be necessary to sometimes view $X$ as ``being over $L$", namely to take the base extension $X_L:=X\times_kL$.
This suggests questions about the existence and uniqueness of $B$-variety structures on $X_L$ extending the given one on $X$.
Actually we will only be interested in ``compatible" $B$-variety structures:

\begin{definition}[Compatible base extension]
\label{compatible}
Suppose $k\subseteq L$ is an extension of bidifferential fields and $X$ is a $B$-variety over $k$.
By a {\em compatible} $B$-variety structure on $X_L$ we mean a biderivation on $L[X_L]=k[X]\otimes_kL$ that extends the given biderivations on $k[X]$ and $L$, and such that $k[X]\subseteq L[X_L]$ is a lifting extension.
\end{definition}

The point is that if we equip $X_L$ with a compatible $B$-variety structure then whenever $Z$ is a $B$-subvariety of $X$,  $Z_L$ is a $B$-subvariety of $X_L$.
As one might by now expect, $X_L$ may not admit a compatible $B$-variety structure, and even if it did there may be no canonical choice.
The following is the geometric formulation of Example~\ref{nonliftex}:

\begin{example}
\label{noncompex}
Let $X$ be the affine line $\mathbb A^1$ over $k$, and let $L$ be the rational function field $k(y)$.
So $X_L$ is the ``generic fibre" of the co-ordinate projection $\mathbb A^2\to\mathbb A^1$.
We put on $k$ and $L$ the trivial biderivation, and on $X$ the trivial $B$-variety structure.
What about $L[X_L]=k(y)[x]$?
We can of course equip it too with the trivial biderivation, making $X_L$ a compatible base extension of $X$.
But we could also consider the biderivation on $L[X_L]$ that is trivial on $k(y)$ and $k[x]$ but where $\{x,y\}=x$.
Then the corresponding $B$-variety structure on $X_L$ is not a compatible base extension of $X$.
Indeed, $Z=\{1\}$ is a $B$-subvariety of $X$ but $Z_L$ is not a  $B$-subvariety of $X_L$, as $(x-1)k(y)[x]$ is a not a bidifferential ideal in $k(y)[x]$ since $\{x-1,y\}=x$.
\end{example}

Fortunately these issues do not arise when taking algebraic base extensions.

\begin{proposition}
\label{alg-bvar}
Suppose $X$ is a $B$-variety over $k$.
Then there is unique $B$-variety structure on $X_{k^{\alg}}$ extending that on $X$, and this is a compatible base extension.
\end{proposition}

\begin{proof}
First, as we are in characteristic zero, there is, by Corollary~\ref{sepcl},  a unique extension of $\bider$ from $k$ to $k^{\alg}$, which we also denote by $\bider$.
It is with respect to $(k,\bider)\subseteq(k^{\alg},\bider)$ that a compatible base extension of $X$ is being claimed.

View the biderivation on $k[X]$ as having values in $k[X]\otimes_kk^{\alg}=k^{\alg}[X_{k^{\alg}}]$.
Fix $b\in k^{\alg}$.
By Theorem~\ref{separable}, applied to $A=k$, $B=k^{\alg}$ and $R=k[X]$, there is an extension of $\bider$ from $k[X]$ to $k[X]\otimes_kk(b)$, and moreover this is a lifting extension.
Now work over $k(b)$ and iterate.
Eventually, by transfinite recursion, we produce a lifting extension $(k[X],\bider)\subseteq (k[X]\otimes_kk^{\alg},\bider')$.
To see that it extends $(k^{\alg},\bider)$ as well, note that $\bider'$ restricts to a $(k[X]\otimes_kk^{\alg})$-valued biderivation on $k^{\alg}$ extending $(k,\bider)$.
The fact that $k^{\alg}$ is differentially trivial over $k$ then implies, by Lemma~\ref{mdt}, that $\bider'$ agrees with $\bider$ on $k^{\alg}$.

So $(k[X]\otimes_kk^{\alg},\bider')$ gives $X_{k^{\alg}}$ a $B$-variety structure that is a compatible base extension of $X$.
Uniqueness follows from Lemma~\ref{mdt} using the fact that $k[X]\otimes_kk^{\alg}$ is differentially trivial over $k[X]$.
\end{proof}

Here is the geometric interpretation of Theorem~\ref{tensor}.

\begin{theorem}
\label{tensor-geometric}
Suppose $\phi:X\to Y$ is a dominant $B$-morphism of irreducible $B$-varieties over $k$.
Then there exists a nonempty Zariski open subset $Y_0\subseteq Y$ and a $B$-variety structure on $X\times Y_0$ such that
\begin{itemize}
\item[(a)]
the co-ordinate projection $X\times Y_0\to X$ is a lifting $B$-morphism,
\item[(b)]
the co-ordinate projection $X\times Y_0\to Y_0$ is a $B$-morphism,
\item[(c)]
$\phi\times\id:X\times Y_0\to Y\times Y_0$ induces a $B$-variety structure on $Y\times Y_0$ such that the co-ordinate projection $Y\times Y_0\to Y$ is a lifting $B$-morphism, and
\item[(d)]
the graph of $\phi\upharpoonright_{Y_0}$ is a $B$-subvariety of $X\times Y_0$.
\end{itemize}
\end{theorem}

\begin{proof}
Let $R:=k[X]$, $S:=k[Y]$, and $\iota:=\phi^*:k[Y]\to k[X]$ which is an embedding as $\phi$ is dominant and is bidifferential as $\phi$ is a $B$-morphism.
Applying Theorem~\ref{tensor} yields a nonzero $f\in k[Y]$ and a biderivation on $k[X]\otimes_kk[Y]_f$, satisfying certain properties.
Now, $k[X]\otimes_kk[Y]_f=k[X\times Y_0]$
where $Y_0:=Y\setminus V(f)$.
This makes $X\times Y_0$ a $B$-variety.
Now the statements~$(a)$ through~$(d)$ of the current theorem are just the geometric meanings of the corresponding statements~$(a)$ through~$(d)$ of Theorem~\ref{tensor}.
\end{proof}

\bigskip
\subsection{Generic fibres and a universal base extension}
With a bit more work we can use Theorem~\ref{tensor-geometric} to find generic fibres of dominant $B$-morphisms.

\begin{definition}
Suppose $\phi:X\to Y$ is a dominant $B$-morphism of irreducible $B$-varieties over $k$.
By a {\em generic $B$-fibre of $\phi$} we mean
\begin{itemize}
\item
a bidifferential field extension $L\supseteq k$,
\item
compatible base extensions $X_L$ and $Y_L$ such that $\phi_L:X_L\to Y_L$ is a $B$-morphism, and
\item a $B$-point $\alpha\in Y_L^\sharp(L)$ that is (Zariski) generic in $Y_L$ over $k$ such that $\phi_L^{-1}(\alpha)$ is a $B$-subvariety of $X_L$
\end{itemize}
We sometimes suppress the data and simply call $X_\alpha:=\phi_L^{-1}(\alpha)$ a generic $B$-fibre over $L$.
\end{definition}

Note, first of all, that fibres over $B$-points need not be $B$-subvarieties in general:

\begin{example}
\label{noPfibrex}
In Example~\ref{nonliftex} (again!) we considered the trivial biderivation on $k[x]$ extended to the biderivation on $k[x,y]$ where $\{x,y\}=x$.
So we have the trivial $B$-variety structure on $\mathbb A^1$ and a $B$-variety structure on $\mathbb A^2$ such that the first co-ordinate projection $\pi:\mathbb A^2\to\mathbb A^1$ is a $B$-morphism.
Now $1\in \mathbb A^1(k)$ is a $B$-point as every $k$-point is.
But $\pi^{-1}(1)$ has vanishing ideal $(x-1)k[x,y]$, which we have seen is not a bidifferential ideal.
\end{example}

\begin{theorem}
\label{genPfibre}
Suppose $\phi:X\to Y$ is a dominant $B$-morphism of irreducible $B$-varieties over $k$.
Then a generic $B$-fibre of $\phi$ exists over $L=k(Y)$.
\end{theorem}

\begin{proof}
Apply Theorem~\ref{tensor-geometric} to get a nonempty Zariski open subset $Y_0\subseteq Y$ and a $B$-variety structure on $X\times Y_0$ such that the co-ordinate projection to $X$ is a lifting $B$-morphism.
At the same time, we get a $B$-variety structure on $Y\times Y_0$ such that $\phi\times\id:X\times Y_0\to Y\times Y_0$ is a $B$-morphism and $Y\times Y_0\to Y$ is a lifting $B$-morphism.

Let $L:=k(Y)=k(Y_0)$.
As localisation is lifting (Theorem~\ref{localisation}), we have that $k[X]\otimes_kk[Y_0]\subseteq k[X]\otimes_kL=L[X_L]$ and $k[Y]\otimes_kk[Y_0]\subseteq k[Y]\otimes_kL=L[Y_L]$ are lifting extension.
In diagrams, we now have:
$$
\xymatrix{
L[Y_L]\ar[rr]^{\phi_L^*}&&L[X_L]\\
k[Y]\otimes_kk[Y_0]\ar[u]\ar[rr]^{(\phi\times\id)^*}&&k[X]\otimes_kk[Y_0]\ar[u]\\
k[Y]\ar[u]\ar[rr]^{\phi^*}&&k[X]\ar[u]
}
$$
with all arrows being bidifferential homomorphisms and the vertical arrows being in addition lifting extensions.
This says precisely that we have compatible $B$-variety structures on the base extensions $X_L$ and $Y_L$ such that $\phi_L$ is a $B$-morphism.

Now, fix $\alpha=(y_1,\dots,y_m)\in Y_L(L)$ such that $y_1,\dots,y_m$ generate $k[Y]$.
So $\alpha$ is a generic point of $Y_L$ over $k$.
Theorem~\ref{tensor-geometric} also tells us that the graph of $\phi$ restricted to $X\times Y_0$ is a $B$-subvariety.
So the ideal $I\subseteq k[X]\otimes_kk[Y_0]$ generated by
$$\{\phi^* y_j\otimes 1-1\otimes y_j:j=1,\dots,m\}$$
is a bidifferential ideal.
Hence the extension  $I\cdot L[X_L]$ is a bidifferential ideal, and this is precisely the vanishing ideal of the fibre $\phi^{-1}(\alpha)$ in $X_L(L)$.
On the other hand, the contraction of $I\cdot L[X_L]$ to $L[Y_L]$, which is precisely the vanishing ideal of $\alpha\in Y_L(L)$, is also bidifferential (being the contraction of a bidifferential ideal under a bidifferential extension).
So $\alpha$ is a $B$-point and the fibre is a $B$-subvariety.
\end{proof}

\begin{remark}
Applying the theorem to the case when $X=Y$ and $\phi=\id$ we see that every $B$-variety has a {\em generic $B$-point} over $k$; namely a compatible base extension to the function field and a $B$-point of this extension that is (Zariski) generic over $k$.
\end{remark}

We can iterate this to get a ``universal" base extension for $\phi:X\to Y$.

\begin{corollary}
\label{universal}
Suppose $\phi:X\to Y$ is a dominant $B$-morphism of absolutely irreducible $B$-varieties over $k$.
There is a bidifferential field extension $k\subseteq L$ and compatible base extensions $X_L$ and $Y_L$ so that $\phi_L:X_L\to Y_L$ is a $B$-morphism, and for every finitely generated subextension field, say  $k\subseteq F\subseteq L$, there is a $B$-point $\alpha\in Y^\sharp_L(L)$ that is generic over $F$ and such that the fibre $\phi_L^{-1}(\alpha)$ is a $B$-subvariety.
\end{corollary}

\begin{proof}
We construct a chain of  bidifferential field extensions
$$k=L_0\subseteq L_1\subseteq L_2\subseteq\cdots$$
with compatible $B$-variety structures on $X_{L_{i}}, Y_{L_{i}}$ such that $\phi_{L_{i}}:X_{L_{i}}\to Y_{L_{i}}$ are $B$-morphisms, and $Y_{L_{i}}(L_{i})$ has a $B$-point that is generic over $L_{i-1}$ and such that the fibre above that point is a $B$-subvariety of $X_{L_{i}}$.
Indeed, we go from $L_i$ to $L_{i+1}:=L_i(Y_{L_i})$ by Theorem~\ref{genPfibre}.
Then $L:=\bigcup_{i<\omega}L_i$ works.
\end{proof}

\begin{definition}[Universal base extension]
Given $\phi:X\to Y$, a dominant $B$-morphism of absolutely irreducible $B$-varieties over $k$, we call the $B$-morphism $\phi_L:X_L\to Y_L$ given by Corollary~\ref{universal} a {\em universal base extension for $\phi:X\to Y$}.
\end{definition}

\bigskip
\section{The bidifferential Dixmier-Moeglin equivalence problem}
\label{sect-bdme}
\noindent
We now extend the Poisson Dixmier-Moeglin equivalence problem to this setting of bidifferential algebra, and make some observations about it.

Fix a bidifferential field $(k,\bider)$ of characteristic zero, and $(R,\bider)$ an affine bidifferential $k$-algebra.
Recall that here affine means that $R$ is an integral domain which is finitely generated as a $k$-algebra.

Given an ideal $I$ of $R$, it is not hard to check that there is a bifferential ideal $J\subseteq I$ that contains all other bidifferential ideals of $I$ (this follows from the fact that if $J_1$ and $J_2$ are bidifferential ideals of $R$ then $J_1+J_2$ is bidifferential).
We call $J$ the {\em bidifferential core} of $I$.

We denote by $\Spec_B(R)$ the space of prime bidifferential ideals of $R$.
We introduce the following three conditions, inspired by the Poisson Dixmier-Moeglin equivalence problem:

\begin{definition}
Let $P\in\Spec_B(R)$. 
\begin{itemize}
\item[(a)]
$P$ is \emph{$B$-locally-closed} if 
$P\neq\displaystyle \bigcap_{P\, \subsetneq \, Q\, \in \, \Spec_B(R)} \; Q$.
\item[(b)]
$P$ is \emph{$B$-primitive} if it is the bidifferential core of a maximal ideal.
\item[(c)]
$P$ is \emph{$B$-rational} if the constants of $\Frac(R/P)$ are contained in $k^{\alg}$. 
\end{itemize}
\end{definition}

The last of these may require some explanation.
As $P$ is a bidifferential ideal we have an induced biderivation on $R/P$, which by Corollary~\ref{fractionext} extends uniquely to the fraction field.
Recall that by the {\em constants} of $\Frac(R/P)$ we mean those $\alpha\in\Frac(R/P)$ such that the associated hamiltonians $\{\alpha,\cdot\}$ and $\{\cdot,\alpha\}$ are identically zero on $\Frac(R/P)$.

Note that when $\bider$ is a Poisson bracket on $R$, these conditions coincide precisely with Poisson locally closed, Poisson primitive, and Poisson rational, considered by, for example, Brown and Gordon~\cite{BrownGordon}.
They are also naturally analogous to the differential version studied in~\cite{BLSM}.

We begin by showing that the expected implications hold also in this setting.
We include proofs for the sake of completeness, though there is no significant difficulty in passing from the Poisson case to the general bidifferential one.

But first a lemma.

\begin{lemma}\label{Bcomponents}
The radical of a bidifferential ideal is bidifferential.
If a radical differential ideal is the intersection of finitely many prime ideals none of which contains the other, then these prime ideals are bidifferential. 
\end{lemma}

\begin{proof}
This follows by the corresponding facts about derivations.
Namely, that in a differential ring the radical of a differential ideal is differential, and if a radical differential ideal is the intersection of finitely many primes, no one containing the other, then those primes are all also differential.
See~\cite[\S0.9 and \S1.2]{kolchin73}.
Now just apply this to each of the hamiltonians $\{r,\cdot\}$ and $\{\cdot, r\}$ for each $r\in R$.
\end{proof}

\begin{proposition}\label{usualimplications}
For any prime bidifferential ideal we have
$$\text{$B$-locally-closed}\implies \text{$B$-primitive},$$
and, if $k$ is contained in the constants of $(R,\bider)$ then we also have
$$\text{$B$-primitive} \implies \text{$B$-rational}.$$
\end{proposition}

\begin{proof}
Working in the quotient ring, it suffices to consider the case of the prime bidifferential ideal $(0)$. 

Assume $(0)$ is $B$-locally-closed. Let
$$J=\bigcap_{(0) \neq \, Q\, \in \, \Spec_B(R)} \; Q.$$
Then $J$ is a nonzero radical bidifferential ideal. Let $\mathfrak m$ be a maximal ideal of $R$ not containing $J$. We claim that $(0)$ is the bidifferential core of $\mathfrak m$ (which shows that $(0)$ is $B$-primitive). For a contradiction, suppose $I$ is a nonzero bidifferential ideal contained in $\mathfrak m$. By Lemma~\ref{Bcomponents}, the radical of $I$ is bidifferential and its prime components are bidifferential. One of these components must be contained in $\mathfrak m$. However, all these components contain $J$, contradicting the fact that $J\not\subseteq \mathfrak m$.

Now assume that $k$ is contained in the constants of $(R,\bider)$.
Suppose that $(0)$ is $B$-primitive.
Namely, there is a maximal ideal $\mathfrak m$ with bidifferential core $(0)$.
Let $r$ be a constant in $(\Frac(R),\bider)$.
We need to show that $r\in k^{\alg}$.

We claim that there are $a,b\in R$ with $r=a/b$ and $b\notin \mathfrak m$.
Assume on the contrary that for every representation of $r=a/b$ we have $b\in \mathfrak m$. 
Our goal is to deduce a contradiction to $(0)$ being the differential core of~$\mathfrak m$ by showing that the bidifferential ideal generated by $b$ is contained in~$\mathfrak m$.
For this it suffices to show that if you iteratively apply hamiltonians to $b$ then you stay in~$\mathfrak m$.
By symmetry we can consider only hamiltonians of the form $\{c,\cdot\}$ with $c\in R$.
Now, since $r$ is a constant, we get
$$0=\{c,r\}=\{c,a/b\}=\frac{\{c,a\}b-a\{c,b\}}{b^2}.$$
Then, either $\{c,b\}=0$ or $r=\frac{a}{b}=\frac{\{c,a\}}{\{c,b\}}$.
In either case, $\{c,b\}\in \mathfrak m$.
Moreover, if $\{c,b\}=0$ then of course further applications of hamiltonians will stay in~$\mathfrak m$, and if $\{c,b\}\neq 0$ then $r=\frac{\{c,a\}}{\{c,b\}}$ and so can repeat the argument with another hamiltonian.
hence, applying iterated hamltonians to $b$ stays in~$\mathfrak m$, as desired.

We can thus write $r=a/b$ such that $b\notin \mathfrak m$.
Work in the localisation $R_b$ on which we can uniquely extend the biderivation by Theorem~\ref{localisation}.
Since $b\notin \mathfrak m$, we have that $\mathfrak m R_b$ is a maximal ideal of $R_b$, and so $R_b/\mathfrak m R_b$ is an algebraic extension of~$k$.
It follows that there is a nonconstant polynomial polynomial $P(t)\in k[t]$ such that $P(r)\in\mathfrak m R_b$.
Note that the constants of $(R_b,\bider)$ contain $k$ and $r$, and hence $\{s,P(r)\}=\{P(r),s\}=0$ for all $s\in R_b$.
It follows that the bidifferential ideal generated by $P(r)$ in $(R_b,\bider)$ is just the principal ideal $P(r)R_b\subseteq\mathfrak m R_b$.
But this implies that the contraction $P(r)R_b\cap R$ is a bidifferential ideal of $(R,\bider)$ contained in~$\mathfrak m$.
As the bidifferential core of $\mathfrak m$ is trivial, we must have $P(r)=0$, proving that $r\in k^{\alg}$.
\end{proof}

We are therefore lead to ask:

\begin{question}[Bidifferential Dixmier-Moeglin Equivalence -- BDME]
\label{BDME}
For which affine bidifferential $k$-algebras is it the case that that every $B$-rational prime bidifferential ideal is $B$-locally closed?
\end{question}

An answer to this question would resolve the as yet open Poisson DME problem.
It is known that the PDME does not hold of all affine Poisson $k$-algebras, see~\cite{BLLSM}.
It is conjectured that the PDME does hold for Poisson-Hopf $k$-algebras, and some partial results have been obtained, see~\cite{LLS}.
The differential analogue of the DME has also been studied, and various results, including its failure in general and its truth in the Hopf case, have been established, see~\cite{BLSM}.

It may be worth pointing out that the following weakening of ``rational implies locally-closed" was shown for affine Poisson algebras in~\cite[Theorem~7.1]{BLLSM}, but the proof goes through more or less verbatum in the general bidifferential setting:

\begin{fact}
\label{bjouanolou}
Suppose $(R,\bider)$ is an affine bidifferential $k$-algebra with constants extending $k$.
If $P\in\Spec_B(R)$ is rational then there are only finitely many prime bidifferential ideals containing $P$ of height equal to $\operatorname{ht}P+1$.
\end{fact}

\bigskip
\subsection{An image and generic fibre approach}
We conclude by illustrating how the constructions of  Section~\ref{sect-balggeom} may be of use in this problem, and why passing to general bidifferential algebras over possibly nonconstant base fields may be useful even if one is primarily interested in the Poisson case.

First, to situate the problem in a geometric setting, let us say for short that an irreducible $B$-variety, $X$, over $k$, is {\em $B$-locally-closed} if in $(k[X],\bider)$ the zero ideal is $B$-locally closed.
Geometrically, this is the property that the union of its proper $B$-subvarieties is not Zariski dense.

\begin{lemma}
An irreducible $B$-variety is $B$-locally-closed if and only if it has only finitely many maximal proper irreducible $B$-subvarieties.
\end{lemma}

\begin{proof}
It is clear that if $X$ has only finitely many maximal proper irreducible $B$-subvarieties then $X$ must be $B$-locally closed, as the union of these would be a proper Zariski closed set containing all the proper $B$-subvarieties.

Suppose $X$ is $B$-locally-closed and let
$$I:=\bigcap\{Q\subseteq k[X]:Q\neq 0\text{ prime bidifferential ideal}\}.$$
Let $P_1,\dots,P_\ell$ be the minimal associated prime ideals of $I$, and denote by $Z_i$ the closed subvariety of $X$ defined by $P_i$.
By Lemma~\ref{Bcomponents}, we get that each $Z_i$ is a $B$-subvariety.
By definition, every proper irreducible $B$-subvariety of $X$ is contained in some $Z_i$.
Since $I$, and hence each $P_i$, is nonzero, we have that $Z_1,\dots,Z_\ell$ are the maximal irreducible proper $B$-subvarieties of $X$.
\end{proof}

The following gives an ``image and generic fibre" characterisation of $B$-locally closed.

\begin{proposition}
\label{fibreimage-lc}
Suppose $\phi:X\to Y$ is a dominant $B$-morphism of irreducible $B$-varieties over $k$.
Assume $Y$ is $B$-locally-closed.
Then the following are equivalent:
\begin{itemize}
\item[(i)]
$X$ is $B$-locally-closed
\item[(ii)]
The generic $B$-fibre of $\phi$ over $k(Y)$ is $B$-locally-closed.
\end{itemize}
\end{proposition}
\begin{proof}
Let $L:=k(Y)$ equipped with the induced bidifferential field structure.
Fix compatible $B$-variety structures on $X_L$ and $Y_L$ such that $\phi_L:X_L\to Y_L$  is a $B$-morphism, and fix $\alpha\in Y_L^\sharp(L)$ generic over $k$ such that $X_\alpha:=\phi_L^{-1}(\alpha)$ is a $B$-subvariety of $X_L$.
This is possible because of Theorem~\ref{genPfibre}.

Note that $k(\alpha)=k(Y)=L$.
The generic $B$-fibre referred to in~(ii) is $X_\alpha$.

We will show, first of all, that there is a bijective containment preserving correspondence between between
$$\mathcal X:=\{\text{irreducible $B$-subvarieties of $X$ over $k$ that map dominantly onto $Y$}\}$$
 and
 
$$\mathcal X_\alpha:=\{\text{irreducible $B$-subvarieties of $X_\alpha$ over $L$}\}$$ given by
$$X'\longmapsto X'_\alpha:=X'_L\cap X_\alpha.$$
That this is true if we replace ``$B$-subvariety" by ``subvariety" follows from the fact that $\alpha$ is generic in $Y$ over $k$.
So what we are claiming is that $X'$ is a $B$-subvariety if and only if $X'_\alpha$ is.
Indeed, suppose $X'\in\mathcal X$.
Then, as the base extension $X_L$ has a compatible $B$-variety structure, $X'_L$ is a $B$-subvariety of $X_L$.
Hence, so is $X'_L\cap X_\alpha=X'_\alpha$.
Conversely, suppose $Z\in\mathcal X_\alpha$.
Let $X'\subseteq X$ be the subvariety over $k$ with $X'_\alpha=Z$.
Note that $I(X')=I(Z)\cap k[X]$.
So, as $k[X]$ is a bidifferential subring of $L[X_L]$, we have that $X'$ is a $B$-subvariety.
We thus have a bijective correspondence as claimed.

Assume now that~(i) holds.
Note that the maximal elements of $\mathcal X$ are also maximal among {\em all} irreducible $B$-subvarieties of $X$ (over $k$), as any $B$-subvariety containing an element of $\mathcal X$ will also map dominantly onto $Y$.
Hence, by~(i), $\mathcal X$ has only finiteley many maximal elements.
By the bijective correspondence, $\mathcal X_\alpha$ has only finitely many maximal elements, thus establishing~(ii).

Now, suppose~(ii) holds so that $\mathcal X_\alpha$ has only finitely many maximal elements.
The bijective correspondence gives us that $\mathcal X$ has only finitely many maximal elements.
Let $X_0$ be the union of these.
So every proper irreducible $B$-subvariety of $X$ that maps dominantly onto $Y$ is contained in $X_0$.
Let $Y_0$ be a proper algebraic subvariety of $Y$ that contains all proper $B$-subvarieties of $Y$; it exists by our assumption on $Y$.
Note that the Zariski closure of the image of a $B$-subvariety under a $B$-morphism is a $B$-subvariety; this is because the intersection of a bidifferential ideal with a bidifferntial subring is a bidifferential ideal of the subring.
Hence, those $B$-subvarieties of $X$ that do not project dominantly onto $Y$ must land in $Y_0$.
It follows that every proper $B$-subvariety of $X$ lands in $X_0\cup\phi^{-1}(Y_0)$.
The latter being a proper subvariety of $X$ by irreducibility, this proves~(i).
\end{proof}

Similarly, we can say that an irreducible $B$-variety over $k$ is {\em $B$-rational} if the zero ideal in $k[X]$ is $B$-rational.
Geometrically, this means that any rational function $f\in k(X)$ such that $\{f,-\}\equiv 0$ on $k(X)$  must be algebraic over $k$.
It is even easier to see that this also has an ``image and generic fibre" characterisation:

\begin{proposition}
\label{fibreimage-r}
Suppose $\phi:X\to Y$ is a dominant $B$-morphism of irreducible $B$-varieties over $k$.
Assume that $Y$ is $B$-rational.
Then the following are equivalent:
\begin{itemize}
\item[(i)]
$X$ is $B$-rational.
\item[(ii)]
The generic $B$-fibre of $\phi$ over $k(Y)$ is $B$-rational.
\end{itemize}
\end{proposition}

\begin{proof}
As before, let $L:=k(Y)$ and use Theorem~\ref{genPfibre} to fix compatible $B$-variety structures on $X_L$ and $Y_L$ such that $\phi_L:X_L\to Y_L$  is a $B$-morphism, and fix $\alpha\in Y_L^\sharp(L)$ generic over $k$ such that $X_\alpha:=\phi_L^{-1}(\alpha)$ is a $B$-subvariety of $X_L$.
So $L=k(\alpha)$ and $X_\alpha$ is our generic $B$-fibre referred to in~(ii).

Note that $k(X)=L(X_\alpha)$.
This immediately shows that~(i) implies~(ii).
For the converse, assume~(ii) and let $f\in k(X)$ be such that $\{f,-\}\equiv 0$.
Then we must have $f\in L$.
But $L=k(Y)$ and so our assumption on $Y$ implies that $f\in k^{\alg}$, establishing~(ii).
\end{proof}

Propositions~\ref{fibreimage-lc} and~\ref{fibreimage-r} suggest an inductive approach to Question~\ref{BDME}.

It is also worth observing that, even if we start with a Poisson variety $X$, the $B$-variety structure we are considering on the generic  fibres $X_\alpha$ in these propositions will not be Poisson and will be over nonconstant base fields.

\appendix

\bigskip
\section{Extending derivations}

\noindent
We review some basic facts about extending derivations.
We could not find good references because we work in the slightly more general, and somewhat unusual, multisorted setting.
That is, we work with the collection of pairs $(R,M)$ where $R$ is a (commutative unitary) ring and $M$ is an $R$-module.
Given such $(R,M)$, we consider $M$-valued derivations on $R$, namely additive maps $d:R\to M$ that satisfy the Leibniz rule $d(rs)=d(r)s+rd(s)$.
By a {\em morphism} $(R,M)\to (R'M')$, we mean a pair $(\phi,\psi)$ where $\phi:R\to R'$ is a ring homomorphism and $\psi:M\to M'$ is a group homomorphism that is $R$-linear in the sense that $\psi(rx)=\phi(r)\psi(x)$, for all $r\in R$ and $x\in M$.
Given derivations $d:R\to M$ and $d:R'\to M'$, a morphism $(R,M)\to (R',M')$ is {\em differential} if $\psi(dr)=d\phi(r)$.
In that case, we say that $d:R'\to M'$ {\em extends} $d:R\to M$.

\begin{definition}[Differential triviality]
\label{dtriv}
An $A$-algebra $R$ is said to be {\em differentially trivial} if $R$ admits no nontrivial $A$-linear derivations (with values in any $R$-module).
\end{definition}

Differential triviality ensures that any derivation on $A$ has at most one extension to $R$.
So, for example, (separably) algebraic field extensions are easily seen to be differentially trivial.
Moreover any derivation on a field does lift to algebraic extensions, so we have existence and uniqueness.
If you consider integral extensions of domains instead of algebraic extensions of fields, one does not quite have differential triviality.
Nevertheless, the following approximation holds:

\begin{fact}
\label{alg-derivation}
Suppose $A\subseteq B$ is an extension of integral domains and $b\in B$ is algebraic over $\Frac(A)$ with the minimal polynomial $P$ of $b$ over $\Frac(A)$ having coefficients in $A$.
Suppose $S$ is a $B$-algebra.
Then any derivation $d:A\to S$ extends uniquely to a derivation $A[b]\to S_f$ where $f:=\frac{dP}{dt}(b)$.
\end{fact}

\begin{proof}
The statement is meant to be understood with respect to the natural morphism $(A,S)\to(A[b],S_f)$.

Uniqueness follows from the fact that $P(b)=0$ forces $d(b)=-\frac{P^d(b)}{f}$, where $P^d\in S[t]$ is obtained by applying $d$ to the coefficients of $P$.

For existence, first we make $S_f$ an algebra over the polynomial ring $A[t]$ by the evaluation map $t\mapsto b$.
Next we extend $d:A\to S$ to a derivation $d:A[t]\to S_f$ by setting $dt:=-\frac{P^d(b)}{f}$.
Note that the transcendentality of $t$ means this is always possible, and that for any $Q\in A[t]$ we have $dQ=\frac{dQ}{dt}(b)dt+Q^d(b)$.
In particular, $dP=0$.
Let $I\subseteq A[t]$ be the ideal of polynomials vanishing at $b$.
As $P\in A[t]$ is monic, and of minimal degree satisfying $P(b)=0$, the division algorithm for monic divisors in $A[t]$ implies that $I=(P)$.
It follows that $dQ=0$ for all $Q\in I$.
So we get the desired induced $S_f$-valued derivation on $A[b]=A[t]/I$.
\end{proof}

\begin{remark}
The above fact only has content when $f$ is not nilpotent, as otherwise the localisation $S_f$ is the trivial ring.
In fact, it is usually applied when $b$ is separably algebraic over $\Frac(A)$ and $S$ has no $B$-torsion, in which case $f$ is not nilpotent.
\end{remark}

Derivations amalgamate nicely with respect to tensor products:

\begin{fact}
\label{tensor-derivation}
Suppose we have a commuting diagram of ring homomorphism
$$
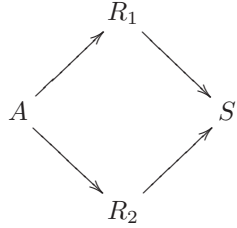
\xymatrix{
&R_1\ar[dr]&\\
A\ar[ur]\ar[dr]&&S\\
&R_2\ar[ur]&
}$$
and derivations $d_1:R_1\to S$ and $d_2:R_2\to S$ that agree on $A$.
Then there is a unique derivation $d:R_1\otimes_AR_2\to S$ extending both $d_1$ and $d_2$.
\end{fact}

\begin{proof}
Define $D:R_1\times R_2\to S$ by $D(x,y)=(d_1x)y+x(d_2y)$.
This is $\mathbb Z$-bilinear and hence induces an additive map $d:R_1\otimes_{\mathbb Z}R_2\to S$, that is easily seen to be a derivation.
Moreover, for each $a\in A$ we have
$$d(a\otimes 1-1\otimes a)=(d_1a)1+a(d_21)-(d_11)a-1(d_2a)=d_1a-d_2a=0$$
since $d_1$ and $d_2$ agree on $A$.
So we obtain an induced derivation $d:R_1\otimes_AR_2\to S$.

The Leibniz rule ensures uniqueness.
\end{proof}


\end{document}